\documentclass[11pt,reqno]{amsart}
\usepackage{graphicx}

\usepackage{amssymb}
\usepackage{cite}
\usepackage{amsmath}
\usepackage{latexsym}
\usepackage{amscd}
\usepackage{amsthm}
\usepackage{mathrsfs}
\usepackage{url}

\usepackage[backref=page]{hyperref}
\numberwithin{equation}{section}

\makeatletter
\@namedef{subjclassname@2020}{%
  \textup{2020} Mathematics Subject Classification}
\makeatother

\newtheorem{thm}{Theorem}
\newtheorem{corr}[thm]{Corollary}
\newtheorem{lem}[thm]{Lemma}
\newtheorem{prop}[thm]{Proposition}

\newtheorem{exam}{Example}

\theoremstyle{definition}
\newtheorem{defn}{Definition}
\theoremstyle{remark}
\newtheorem{rem}[thm]{Remark}

\def\R{\mathbb R}

\def\SS{\mathbb S}

\def\pt{\partial}

\DeclareMathOperator*{\essinf}{ess\,inf}

\begin{document}
\title[Sharp bounds for the first two eigenvalues]{Sharp bounds for the first two eigenvalues of an exterior Steklov eigenvalue problem}
\author[C.~Xiong]{Changwei~Xiong}
\address{School of Mathematics, Sichuan University, Chengdu 610065, Sichuan,  P.~R.~China}
\email{\href{mailto:changwei.xiong@scu.edu.cn}{changwei.xiong@scu.edu.cn}}
\date{\today}
\thanks{This research was supported by National Key R and D Program of China 2021YFA1001800, NSFC Grant no.~12171334, and the funding (no.~1082204112549) from Sichuan University.}

\maketitle

\begin{abstract}
Let $U\subset \R^n$ ($n\geq 3$) be an exterior Euclidean domain with smooth boundary $\pt U$. We consider the Steklov eigenvalue problem on $U$. First we derive a sharp lower bound for the first eigenvalue in terms of the support function and the distance function to the origin of $\pt U$. Second under various geometric conditions on $\pt U$ we obtain sharp upper bounds for the first eigenvalue. Along the proof, we get a sharp upper bound for the capacity of $\pt U$ when $n=3$ and $\pt U$ is connected. Last we also discuss an upper bound for the second eigenvalue.
\end{abstract}

\section{Introduction}\label{sec1}

The investigation on various eigenvalue problems is one of the most important and extensively-studied topics in the fields of differential geometry, partial differential equations, etc. See e.g. the excellent surveys \cite{Pay67,AB07,BLL12,GP14} on different types of eigenvalue problems. In this paper we are concerned with an exterior Steklov eigenvalue problem in the Euclidean space $\R^n$ ($n\geq 3$). The classical (interior) Steklov eigenvalue problem has received considerable attention since it was introduced by Steklov \cite{Ste02} around 1900; see \cite{KKK14} for a historical introduction and \cite{GP14} for a specialized review. In contrast, the exterior Steklov eigenvalue problem has not been much studied; see e.g. \cite{Pay56,AH14} for some results on it. However, besides differential geometry and partial differential equations, the exterior Steklov eigenvalue problem also plays an indispensable role in potential theory, mathematical physics, functional analysis, etc. So we believe it is desirable to contribute and draw more attention to this eigenvalue problem.

To describe the exterior Steklov eigenvalue problem let us first set the context. Let $U\subset \R^{n}$ ($n\geq 3$) be an exterior domain in the Euclidean space $\R^{n}$. Namely, $U$ is a non-empty connected open set in $\R^n$ such that $\R^n\setminus U$ is non-empty and compact. Suppose further that the origin $O\in \R^n\setminus \overline{U}$ and the boundary $\pt U$ is the union of finitely many disjoint, closed, Lipschitz hypersurfaces, each of finite hypersurface area. 

We consider the following exterior Steklov eigenvalue problem:
\begin{equation}\label{eq-problem}
\begin{cases}
\Delta \varphi=0,\text{ in }U,\\
-\dfrac{\pt \varphi}{\pt \nu}=\xi \varphi, \text{ on }\pt U,
\end{cases}
\end{equation}
where $\nu$ is the unit normal vector along $\pt U $ pointing into $U$ and $\varphi$ belongs to the space $E^1(U)$ of functions on $U$ having finite energy. We refer to Section~\ref{sec2} for the precise definition for $E^1(U)$. This eigenvalue problem is well-posed and has a discrete spectrum:
\begin{equation*}
0<\xi_1\leq \xi_2\leq \cdots \nearrow +\infty.
\end{equation*}
The variational characterization for $\xi_i$ $(i\geq 1)$ reads
\begin{equation}\label{v-c}
\xi_i=\inf_{\substack{\varphi\in E^1(U)\\\int_{\pt U} \varphi \varphi_jda=0,\; j=1,2,\dots,i-1.}}\frac{\int_{U}|\nabla \varphi|^2 dx}{\int_{\pt U} \varphi^2da},
\end{equation}
where $\varphi_j$ ($j=1,2,\dots,i-1$) are the first $i-1$ eigenfunctions. For the analytic setting and basic properties of the eigenvalue problem \eqref{eq-problem}, we refer to the work \cite{AH14}.

In this paper we first derive the following sharp lower bound for the first eigenvalue $\xi_1$.
\begin{thm}\label{thm0}
Let $U\subset \R^{n}$ ($n\geq 3$) be an exterior domain in the Euclidean space $\R^{n}$ with $C^1$ boundary. Suppose the origin $O\in \R^n\setminus \overline{U}$. Then there holds
\begin{equation}\label{bound0}
\xi_1\geq (n-2)\min_{\pt U} \frac{\langle x,\nu\rangle}{|x|^2},
\end{equation}
with equality if and only if $U$ is the exterior domain of some centered ball $B_R$ ($R>0$), i.e., $U=\R^n\setminus \overline{B_R}$.
\end{thm}
\begin{rem}
The bound \eqref{bound0} is meaningful only when the boundary $\pt U$ is star-shaped with respect to the origin, i.e., $\langle x,\nu\rangle >0$ on $\pt U$. And if the boundary is only Lipschitz, the lower bound becomes (see Theorem~\ref{thm9} below)
\begin{equation*}
\xi_1\geq (n-2)\essinf_{\pt U} \frac{\langle x,\nu\rangle}{|x|^2}.
\end{equation*}
\end{rem}
In view of the variational characterization \eqref{v-c} for $\xi_1$, we get the following Poincar\'{e}--trace inequality.
\begin{corr}
Assumptions are as in Theorem~\ref{thm0}. There holds
\begin{equation}
\int_{U}|\nabla \varphi|^2 dx\geq (n-2)\min_{\pt U} \frac{\langle x,\nu\rangle}{|x|^2}\int_{\pt U} \varphi^2da,\quad \varphi\in E^1(U),
\end{equation}
with equality if and only if $U$ is the exterior domain of some centered ball $B_R$ ($R>0$), i.e., $U=\R^n\setminus \overline{B_R}$, and $\varphi$ is its first exterior Steklov eigenfunction.
\end{corr}
To the best of our knowledge, Theorem~\ref{thm0} is the first lower bound for the first exterior Steklov eigenvalue. We note a remark by L.~E.~Payne in \cite{Pay67}, ``Upper bounds for physically interesting eigenvalues are usually not difficult to obtain, but in most cases lower bounds are far more important'' (Page~461 in \cite{Pay67}). In light of his remark, we hope our Theorem~\ref{thm0} would motivate more studies on the lower bounds of the exterior Steklov eigenvalues. As for the proof of Theorem~\ref{thm0}, our method is inspired by the work \cite{Pro60} where nice lower bounds for the first Dirichlet eigenvalue of elliptic operators were derived. Two main ingredients in our proof are the variational characterization \eqref{v-c} for $\xi_1$ and a suitably chosen vector field on $\overline{U}$.

Second we shall obtain the following sharp upper bounds for the first eigenvalue $\xi_1$.
\begin{thm}\label{thm1}
Let $U\subset \R^n$ ($n\geq 3$) be an exterior domain with smooth boundary.
\begin{enumerate}
  \item If $\pt U$ is convex, then
  \begin{align}\label{bound1}
\xi_1\leq  \frac{1}{|\pt U|}\frac{1}{\int_0^\infty (\sum_{i=0}^{n-1}\int_{\pt U} \sigma_i da \cdot t^i)^{-1}dt},
\end{align}
where $\sigma_i$ ($i=1,2,\dots,n-1$) denotes the $ith$ mean curvature of the boundary (e.g., $\sigma_1$ is the summation of the principal curvatures of the boundary) and $\sigma_0=1$ by convention.
\item If $\pt U$ is star-shaped with respect to the origin (i.e., the support function $\langle x,\nu\rangle>0$), then
		\begin{equation}\label{bound2}
			\xi_1\leq   \frac{n-2}{|\pt U|}\int_{\pt U} \langle x,\nu\rangle^{-1}da.
		\end{equation}
  \item If $\pt U$ is mean convex and outer-minimizing (i.e., $\pt U$ minimizes area among all hypersurfaces homologous to $\pt U$ in $U$), then
  \begin{align}\label{bound3}
  \xi_1\leq \frac{n-2}{(n-1)|\pt U|}\int_{\pt U}\sigma_1 da.
  \end{align}
  \item If $\pt U$ is mean convex and star-shaped, then
  \begin{align}\label{bound4}
  \xi_1\leq \frac{n-2}{(n-1)|\pt U|}\int_{\pt U}\sigma_1 da.
  \end{align}
  \item If $n=3$ and $\Sigma=\pt U$ is connected, then
  \begin{align}\label{bound5}
  \xi_1&\leq \sqrt{\frac{4\pi}{|\Sigma|}} \frac{\sqrt{\int_{\Sigma}H^2 da/(16\pi)-1}}{\mathrm{arsinh}\sqrt{\int_{\Sigma}H^2 da/(16\pi)-1}},\quad H:=\sigma_1.
  \end{align}
  \item If $\pt U$ is only smooth, then
   \begin{align}\label{bound6}
   \xi_1\leq (n-2) \left(\frac{\int_{\pt U}|\sigma_1/(n-1)|^{(2n-3)/(n-1)}da}{|\pt U|}\right)^{(n-1)/(2n-3)}.
   \end{align}
\end{enumerate}
Moreover, the equality holds in \eqref{bound1}, \eqref{bound3}, \eqref{bound4}, \eqref{bound5} or \eqref{bound6} if and only if $U$ is the exterior domain of a round ball; the equality holds in \eqref{bound2} if and only if $U$ is the exterior domain of a round ball centered at the origin.
\end{thm}
The proof for Theorem~\ref{thm1} is the combination of the observation
\begin{equation}
\xi_1\leq \frac{\mathrm{Cap}(\pt U)}{|\pt U|}
\end{equation}
by Payne \cite{Pay56}, and various estimates for the electrostatic capacity $\mathrm{Cap}(\pt U)$ for the boundary $\pt U$. Here the electrostatic capacity of $\pt U$ is defined as
\begin{align*}
\mathrm{Cap}(\pt U):&=\inf \left\{ \int_{\R^n} |\nabla f|^2 dx: f \in C_c^{\infty}(\R^n), f \geq 1\text{ on }\R^n\setminus U \right\}\\
&=\int_U |\nabla u|^2 dx,
\end{align*}
where the infimum is achieved by the so-called electrostatic capacitary potential $u\in E^1(U)$, i.e., the unique solution in $E^1(U)$ of the PDE
\begin{equation*}
\begin{cases}
\Delta u=0,\text{ in }U,\\
u=1, \text{ on }\pt U.
\end{cases}
\end{equation*}
We collect known estimates for $\mathrm{Cap}(\pt U)$ in Theorem~\ref{thm10} of Section~\ref{sec4}. Our new estimate for $\mathrm{Cap}(\pt U)$ is in Theorem~\ref{thm12}. For the proof of Theorem~\ref{thm12}, we employ a classical approach which may go back to \cite{PS51,Sze31} and has been applied successfully in some papers, e.g., in \cite{FS14,Xiao16,Xiao17,BM08,LXZ11,Xiao17a}. In particular, we mainly follow the argument in Bray and Miao's \cite{BM08} where the weak inverse mean curvature flow developed by Huisken and Ilmanen and the monotonicity of the Hawking mass of the evolving surfaces along this flow \cite{HI01} comprise two of the key tools. The proof of our Theorem~\ref{thm12} relies on the introduction of a modified Hawking mass of the evolving surfaces which is inspired by the work \cite{HW15}.

Next, for the second exterior Steklov eigenvalue $\xi_2$, we obtain the following upper bound by use of good test functions in the min-max variational characterization for $\xi_2$. Our result generalizes the corresponding one in Payne's \cite{Pay56} for the $3$-dimensional case to the higher dimensional case.
\begin{thm}\label{thm2}
Let $U\subset \R^n$ ($n\geq 3$) be an exterior domain with smooth boundary. Assume that the virtual mass potential $w[e]$ (see \eqref{v-m-p} for its definition) in the direction $e\in \SS^{n-1}$ satisfies
\begin{equation}\label{eq-assumption}
\int_{\pt U}w[e]da=0,\quad \forall e\in \SS^{n-1}.
\end{equation}
Then we have
\begin{equation}\label{bound7}
\xi_2\leq \max \left\{\frac{ \mathrm{Cap}(\pt U)}{|\pt U|},\frac{(n-1)|\pt U|}{nV}\right\},
\end{equation}
where $\mathrm{Cap}(\pt U)$ is the electrostatic capacity of $\pt U$ and $V$ denotes the volume of the set $\R^n\setminus U$.
\end{thm}
\begin{rem}
The virtual mass corresponding to the domain $U$ is an important physical quantity; see e.g. \cite{SS49,AK07} for an introduction. The assumption \eqref{eq-assumption} indicates some kind of symmetry for the domain $U$. We do not know any characterization of domains satisfying it. But all centrally symmetric domains (with respect to some point) do fulfill the assumption; see Remark~\ref{example} below for the demonstration. See also Remark~\ref{rem-Payne} below for another comment on it.
\end{rem}
When $\pt U$ is mean convex and star-shaped, it was proved in \cite{Xiao17} (see Theorem~\ref{thm10} in Section~\ref{sec4} below) that
\begin{equation}\label{eq-capacity}
\mathrm{Cap}(\pt U)\leq \frac{n-2}{n-1}\int_{\pt U} \sigma_1 da.
\end{equation}
If we assume further that $\pt U$ is convex, then we are allowed to use the Alexandrov--Fenchel inequality (see \cite{Sch14})
\begin{equation}\label{eq-AF}
\frac{n}{n-1}\int_{\pt U} \sigma_1 da\leq \frac{|\pt U|^2}{V}.
\end{equation}
Therefore by combining \eqref{eq-capacity}, \eqref{eq-AF}, Theorem~\ref{thm2} and Theorem~\ref{thm1}, we immediately conclude the following sharp upper bounds for $\xi_1$ and $\xi_2$.
\begin{corr}\label{corr1}
Let $U\subset \R^n$ ($n\geq 3$) be an exterior domain with smooth boundary. Assume $\pt U$ is convex. Then we have
\begin{equation}\label{eq-upper1}
\xi_1\leq \frac{(n-2)|\pt U|}{nV}.
\end{equation}
Suppose additionally that the condition \eqref{eq-assumption} holds. Then we have
\begin{equation}\label{eq-upper2}
\xi_2\leq \frac{(n-1)|\pt U|}{nV}.
\end{equation}
Both equalities in \eqref{eq-upper1} and \eqref{eq-upper2} hold if $U$ is the exterior domain of a round ball.
\end{corr}
\begin{rem}
When $n=3$, Corollary~\ref{corr1} can be obtained by combining Payne's work \cite{Pay56} and Schiffer's work \cite{Sch57}. But for $n\geq 4$, new techniques, e.g. the inverse mean curvature flow (see \cite{Ger90,Urb90}), are required first to derive \eqref{eq-capacity}, in order to prove Corollary~\ref{corr1}.
\end{rem}
\begin{rem}\label{rem-conj}
It is worth mentioning that \eqref{eq-upper1} and \eqref{eq-upper2} were conjectured to be true for any smooth exterior domain $U$ by Payne in \cite{Pay56}.
\end{rem}

Last, we give some remarks on the problem to get sharp upper bounds for $\xi_1$, towards the conjecture mentioned in Remark~\ref{rem-conj}. The natural idea is to use suitable test functions in the variational characterization for $\xi_1$, i.e.,
\begin{equation}\label{eq-RQ}
\xi_1\leq \frac{\int_{U}|\nabla f|^2 dx}{\int_{\pt U} f^2da},\quad f\in E^1(U).
\end{equation}
Let $\Omega:=\R^n\setminus \overline{U}$. Assume the origin $O\in \Omega$. Then the first candidate is
\begin{equation*}
f_1(x)=\frac{1}{|x|^{n-2}}.
\end{equation*}
The second candidate is the gravitational potential for $\Omega$, i.e.,
\begin{equation*}
f_2(x)=-\frac{1}{(n-2)\omega_{n-1}}\int_\Omega \frac{1}{|y-x|^{n-2}}dy,
\end{equation*}
where $\omega_{n-1}=|\SS^{n-1}|$. The third candidate is the single layer potential for $\pt \Omega$, i.e.,
\begin{equation*}
f_3(x)=-\frac{1}{(n-2)\omega_{n-1}}\int_{\pt \Omega} \frac{1}{|y-x|^{n-2}}dy.
\end{equation*}
Note that all $f_i$ ($i=1,2,3$) are harmonic outside $\Omega$ and have the right decay rate at $\infty$. Moreover, when $\Omega$ is a Euclidean ball, all of them are indeed the first Steklov eigenfunctions. So the question is how to estimate the Rayleigh quotient in \eqref{eq-RQ} for $f_i$ ($i=1,2,3$).

The organization of the paper is as follows. In Section~\ref{sec2} we review the setting of the exterior Steklov eigenvalue problem and discuss the case of the exterior domain of a round ball. In Section~\ref{sec3} we prove the lower bound for the first eigenvalue. In Section~\ref{sec4} we obtain the various upper bounds for the first eigenvalue. In the next Section~\ref{sec5} we discuss an upper bound for the second eigenvalue. In the last section we present some auxiliary and related results used in the paper.

\section{Preliminaries}\label{sec2}

For the setting of the exterior Steklov eigenvalue problem, we mainly follow the work \cite{AH14}.

Let $U$ be an exterior domain in $\R^n$ ($n\geq 3$). In other words, $U$ is a non-empty, open, connected set in $\R^n$ such that its complement $\R^n\setminus U$ is a non-empty compact set. Moreover, we assume that the origin $O\in \Omega:=\R^n\setminus \overline{U}$ and the boundary $\pt U$ is the union of finitely many disjoint, closed, Lipschitz hypersurfaces, each of finite hypersurface area. Let $u$ be a Lebesgue measurable extended real-valued function on $U$. We say that $u$ \emph{decays at infinity} if for each $c>0$, the set
\begin{equation*}
S_c(u):=\{x\in U:|u(x)|\geq c\}
\end{equation*}
has finite Lebesgue measure. Furthermore, we say that $u$ has \emph{finite energy}, if $u$ decays at infinity, $u\in L^1(U_R)$ for each $R>R_b$ and $|\nabla u|\in L^2(U)$. Here $U_R:=U\cap B_R$ and $R_b:=\sup \{|x|:x\in \pt U\}$ with $B_R$ denoting the Euclidean ball of radius $R$. Now we define $E^1(U)$ to be the set of all functions having finite energy on $U$. The set $E^1(U)$ is the main space of functions we will use in this paper.

Next we introduce the harmonic function in $E^1(U)$. A function $u\in E^1(U)$ is called \emph{harmonic} if it satisfies
\begin{equation*}
\int_U \nabla u\cdot \nabla \varphi dx =0
\end{equation*}
for all $\varphi \in C_c^1(U)$. Let $\mathcal{H}(U)$ be the subspace of $E^1(U)$ of all harmonic functions. In \cite[Section~12]{AH14} and \cite[Section~5]{AH14}, respectively, the existence and uniqueness of solutions to a harmonic Neumann problem and a harmonic Dirichlet problem, respectively, on $U$ is proved. For later use, we state them here as follows.
\begin{prop}[Thm.~12.1 in \cite{AH14}]\label{N-p}
For any $\eta\in L^q(\pt U)$ with $q\geq 2(n-1)/n$, there exists a unique $u\in \mathcal{H}(U)$ such that $\nabla_\nu u=\eta$ on $\pt U$.
\end{prop}
\begin{prop}[Thm.~5.1 in \cite{AH14} and the paragraph following it]\label{D-p}
For any $f\in H^{1/2}(\pt U)$, there exists a unique $u\in \mathcal{H}(U)$ such that $u=f$ in the sense of trace on $\pt U$.
\end{prop}
Here the fractional Sobolev space $H^{1/2}(\pt U)$ is defined to be the set of functions $f\in L^2(\pt U)$ satisfying
\begin{equation*}
\int_{\pt U}\int_{\pt U}\frac{|f(x)-f(y)|^2}{|x-y|^n} da(x)da(y)<+\infty.
\end{equation*}

Thanks to the work \cite{AH14}, the eigenvalues of the exterior Steklov problem admit the following variational characterization
\begin{equation*}
\xi_i=\inf_{\substack{\varphi\in E^1(U)\\ \int_{\pt U} \varphi \varphi_jda=0,\; j=1,2,\dots,i-1}}\frac{\int_{U}|\nabla \varphi|^2 dx}{\int_{\pt U} \varphi^2da},\quad i\geq 1,
\end{equation*}
where $\varphi_j$ ($j=1,2,\dots,i-1$) are the first $i-1$ eigenfunctions.

Next we discuss the exterior Steklov eigenvalue problem on the special domain, the exterior domain of a Euclidean ball.
\begin{exam}
For the case of the Euclidean ball, i.e. $\R^n\setminus \overline{U}=B_R$ for some $R>0$, the eigenvalues and the corresponding eigenfunctions can be computed explicitly. In fact, by separation of variables, any exterior Steklov eigenfunction $\varphi$ can be expressed as $r^{2-n-m}\cdot \omega_m(p)$ in the polar coordinates, where $\omega_m(p)$ ($p\in \SS^{n-1}$) is a spherical harmonic on $\SS^{n-1}$ of degree $m\geq 0$.

For readers' convenience, let us recall some knowledge on spherical harmonics. Given a spherical harmonic $\omega_m$ on $\SS^{n-1}$ of degree $m\geq 0$, it can be regarded as the restriction on $\SS^{n-1}$ of a harmonic homogeneous polynomial $\tilde{\omega}_m$ on $\R^n$ of the same degree $m$. For each $m\geq 0$, let $\mathcal{D}_m$ be the space of harmonic homogeneous polynomials on $\R^n$ of degree $m$ and $\mu_m$ be the dimension of $\mathcal{D}_m$. For example, we know
\begin{align*}
&\mathcal{D}_0=span\{1\},\quad \mu_0=1,\\
&\mathcal{D}_1=span\{x_i,\:i=1,2,\dots,n\},\quad \mu_1=n,\\
&\mathcal{D}_2=span\{x_ix_j,\: x_1^2-x_k^2,\:1\leq i<j\leq n,\: 2\leq k\leq n\},\quad \mu_2=\frac{n^2+n-2}{2},
\end{align*}
and $\mu_m=C_{n+m-1}^{n-1}-C_{n+m-3}^{n-1}$ for $m\geq 2$. See \cite{ABR92} for basic facts concerning $\mathcal{D}_m$ and $\mu_m$. For a spherical harmonic $\omega_m$ on $\SS^{n-1}$ of degree $m\geq 0$, one of its basic properties is that $-\Delta_{\SS^{n-1}}\omega_m=\tau_m\omega_m$ with $\tau_m=m(n-2+m)$.

Therefore, back to our exterior Steklov problem for $U=\R^n\setminus \overline{B_R}$, we can list its first $n+1$ eigenvalues and their corresponding eigenfunctions as follows:
\begin{align*}
&\xi_1=\frac{n-2}{R},\quad \varphi_1(x)=\frac{1}{|x|^{n-2}},\\
&\xi_2=\cdots=\xi_{n+1}=\frac{n-1}{R}, \quad \varphi_i(x)=\frac{x_{i-1}}{|x|^n},\quad i=2,3,\dots,n+1.
\end{align*}

\end{exam}

Using the above example, we can get the following useful fact concerning the decay rate of functions in $
\mathcal{H}(U)$.
\begin{prop}\label{prop-decay-rate}
Let $u\in \mathcal{H}(U)$. Then we have the decay rate for $u(x)$ and its derivatives
\begin{equation}
u(x)=O(|x|^{2-n}),\quad \nabla^k u(x)=O(|x|^{2-n-k}),\ k\geq 1,\text{ as } x\to \infty.
\end{equation}
\end{prop}
\begin{proof}
Take a large ball $B_R$ containing the boundary $\pt U$ of the domain $U$. As in the above example, let $\{\varphi_i(x)\}_{i=1}^\infty$ be the family of the exterior Steklov eigenfunctions for $\R^n\setminus \overline{B_R}$ such that its restriction on $\pt B_R$ forms an orthonormal basis of $L^2(\pt B_R)$, i.e.,
\begin{equation}
\int_{\pt B_R} \varphi_i\varphi_jda=\delta_{ij},\quad i,j=1,2,\dots.
\end{equation}
Then we can decompose the restriction on $\pt B_R$ of $u\in \mathcal{H}(U)$ as
\begin{equation*}
u|_{\pt B_R}=\sum_{k\geq 1}c_k\varphi_k|_{\pt B_R},\quad c_k\in \R,\quad k\geq 1.
\end{equation*}
We claim that
\begin{equation*}
u =\sum_{k\geq 1}c_k\varphi_k,\text{ in }\R^n\setminus B_R.
\end{equation*}
In fact, both $u$ and $\sum_{k\geq 1}c_k\varphi_k$ lie in $\mathcal{H}(\R^n\setminus \overline{B_R})$, and they admit the same boundary value on $\pt B_R$. Then Proposition~\ref{D-p} implies the claim.

By the claim and the expressions of $\varphi_i$ ($i\geq 1$) in the above example, the conclusion on $u(x)$ follows. Then term-by-term differentiation of the expansion for $u$ yields the conclusion on $\nabla^k u(x)$ for $k\geq 1$.
\end{proof}

\section{Proof of Theorem~\ref{thm0}}\label{sec3}

In this section we first prove a more general result.
\begin{thm}\label{thm9}
Let $U\subset \R^{n}$ ($n\geq 3$) be an exterior domain in the Euclidean space $\R^{n}$ with Lipschitz boundary. Let $\xi_1$ be the first exterior Steklov eigenvalue and $\varphi$ a corresponding eigenfunction. Assume that the vector field $\vec{P}(x)$ consisting of $n$ smooth functions $P_i(x)$, $i=1,2,\dots,n$, on $\overline{U}$ satisfies
\begin{equation}\label{eq-condition}
\lim_{R\to \infty}\int_{\pt B_R}\langle \varphi^2\vec{P},\nu \rangle da=0,
\end{equation}
where $B_R$ denotes the centered ball with radius $R$, and
\begin{equation}\label{eq-M}
\mathrm{div}\vec{P}-|\vec{P}|^2\geq 0.
\end{equation}
Then we have
\begin{equation*}
\xi_1\geq \essinf_{\pt U} \langle \vec{P},\nu\rangle.
\end{equation*}
\end{thm}
\begin{proof}
First we note
\begin{align*}
\xi_1\int_{\pt U}\varphi^2da=\int_{U}|\nabla \varphi|^2dx.
\end{align*}
For the vector field $\vec{P}(x)$ on $\overline{U}$, we get by use of the divergence theorem
\begin{align*}
\xi_1&\int_{\pt U}\varphi^2da-\int_{\pt U} \varphi^2 \langle \vec{P},\nu\rangle da=\int_{U}|\nabla \varphi|^2dx-\int_{\pt U} \varphi^2 \langle \vec{P},\nu\rangle da\\
&=\int_{U}|\nabla \varphi|^2dx+\int_{B_R\setminus \overline{\Omega}} \mathrm{div}(\varphi^2\vec{P})dx-\int_{\pt B_R}\langle \varphi^2\vec{P},\nu \rangle da,
\end{align*}
where $B_R$ is a ball containing $\overline{\Omega}$ (recall $\Omega=\R^n\setminus \overline{U}$).

Using \eqref{eq-condition} and the Cauchy--Schwarz inequality, we see
\begin{align*}
\xi_1\int_{\pt U}\varphi^2da-\int_{\pt U} \varphi^2 \langle \vec{P},\nu\rangle da&=\int_{U}\left(|\nabla \varphi|^2+2\sum_i P_{i}\varphi_i\varphi+\mathrm{div}\vec{P}\cdot \varphi^2\right)dx\\
&\geq \int_U\left(\mathrm{div}\vec{P}-|\vec{P}|^2\right)\varphi^2 dx.
\end{align*}
Finally noting the condition \eqref{eq-M}, we finish the proof.
\end{proof}

Now for the proof of Theorem~\ref{thm0}, we choose
\begin{equation*}
P_i(x)=(n-2)\frac{x_i}{|x|^2},\quad i=1,2,\dots,n.
\end{equation*}
Then for the condition \eqref{eq-condition}, we see
\begin{align*}
\int_{\pt B_R}\langle \varphi^2\vec{P},\nu \rangle da&=(n-2)\frac{1}{R}\int_{\pt B_R}\varphi^2da\to 0, \text{ as }R\to \infty,
\end{align*}
since $\varphi(x)=O(|x|^{2-n})$ as $x\to \infty$ in view of Proposition~\ref{prop-decay-rate}.

Next the condition \eqref{eq-M} becomes
\begin{equation*}
\mathrm{div}\vec{P}-|\vec{P}|^2
=\frac{(n-2)^2}{|x|^2}-\frac{(n-2)^2}{|x|^2}= 0.
\end{equation*}

So we get the inequality in Theorem~\ref{thm0}.

Now assume the equality holds. Then from the proof above, we see that $\nabla \varphi=-\varphi \vec{P}$ and $\langle x,\nu\rangle/|x|^2$ is a constant function along the boundary $\pt U$. Then $\varphi(x)=c|x|^{2-n}$, and along $\pt U$ the radial function $|x|$ is constant (by considering the maximum and minimum points of $|x|$ on the $C^1$ boundary $\pt U$). So $U$ is the exterior domain of a centered ball, and we finish the proof of Theorem~\ref{thm0}.

\section{Estimates for the electrostatic capacity and proof of Theorem~\ref{thm1}}\label{sec4}

For the proof of Theorem~\ref{thm1}, we employ the following observation by Payne \cite{Pay56}:
\begin{equation}
\xi_1\leq \frac{\mathrm{Cap}(\pt U)}{|\pt U|}.
\end{equation}
So it suffices to get nice upper bounds for $\mathrm{Cap}(\pt U)$. For that purpose we recall the following results.

\begin{thm}\label{thm10}
Let $U\subset \R^n$ ($n\geq 3$) be an exterior domain with smooth boundary.
\begin{enumerate}
  \item If $\pt U$ is convex, then (\cite{Sze31,PS51,Xiao17a,LX22})
  \begin{align*}
\mathrm{Cap}(\pt U)\leq  \frac{1}{\int_0^\infty (\sum_{i=0}^{n-1}\int_{\pt U} \sigma_i da \cdot t^i)^{-1}dt},
\end{align*}
with equality if and only if $U$ is the exterior domain of a round ball.
\item If $\pt U$ is star-shaped with respect to the origin, then (\cite{PS51,LXZ11,LX22})
		\begin{equation}
			\mathrm{Cap}(\pt U)\leq   (n-2)\int_{\pt U} \langle x,\nu\rangle^{-1}da,
		\end{equation}
with equality if and only if $U$ is the exterior domain of a round ball centered at the origin.
  \item If $\pt U$ is mean convex and outer-minimizing, then (\cite[Theorem~2~(a)]{FS14})
  \begin{align*}
  \mathrm{Cap}(\pt U)\leq \frac{n-2}{n-1}\int_{\pt U}\sigma_1 da,
  \end{align*}
  with equality if and only if $U$ is the exterior domain of a round ball.
  \item If $\pt U$ is mean convex and star-shaped, then (\cite[Theorem~3.1]{Xiao17})
  \begin{align*}
  \mathrm{Cap}(\pt U)\leq \frac{n-2}{n-1}\int_{\pt U}\sigma_1 da,
  \end{align*}
  with equality if and only if $U$ is the exterior domain of a round ball.
  \item If $n=3$ and $\pt U$ is connected, then (\cite[Corollary~2]{BM08})
  \begin{align}\label{eq-BM08}
  \mathrm{Cap}(\pt U)\leq \sqrt{\pi}\sqrt{|\pt U|}\left(1+\sqrt{\frac{\int_{\pt U}H^2da}{16\pi}}\right),\quad H:=\sigma_1,
  \end{align}
  with equality if and only if $U$ is the exterior domain of a round ball.
  \item If $\pt U$ is only smooth, then (\cite[Corollary~4.4]{AM20})
   \begin{align*}
   \mathrm{Cap}(\pt U)\leq (n-2)|\pt U|\left(\frac{\int_{\pt U}|\sigma_1/(n-1)|^{(2n-3)/(n-1)}da}{|\pt U|}\right)^{(n-1)/(2n-3)},
   \end{align*}
   with equality if and only if $U$ is the exterior domain of a round ball.
\end{enumerate}
\end{thm}
\begin{rem}
By checking the proof of \cite[Corollary~4.4]{AM20}, we may conclude
\begin{align*}
\mathrm{Cap}(\pt U)\leq \frac{n-2}{n-1}|\pt U|\max_{\pt U} \sigma_1,
\end{align*}
with the rigidity statement. Note that here the upper bound involves $\max_{\pt U} \sigma_1$, not $\max_{\pt U} |\sigma_1|$.
\end{rem}

In this paper, we provide a new sharp upper bound for $\mathrm{Cap}(\pt U)$ in the case that $n=3$ and $\pt U$ is connected. Compared with \eqref{eq-BM08}, our upper bound \eqref{eq-new} below is better, since there holds the elementary inequality
\begin{equation}
2\frac{\sqrt{s}}{\mathrm{arsinh}\sqrt{s}}<1+\sqrt{s+1},\quad \forall s>0.
\end{equation}
Precisely, we prove the following.
\begin{thm}\label{thm12}
Let $U\subset \R^3$ be an exterior domain with smooth boundary. If $\Sigma=\pt U$ is connected, then
\begin{align}\label{eq-new}
\mathrm{Cap}(\pt U)&\leq 2(\pi|\Sigma|)^{1/2} \frac{\sqrt{\int_{\Sigma}H^2 da/(16\pi)-1}}{\mathrm{arsinh}\sqrt{\int_{\Sigma}H^2 da/(16\pi)-1}},
\end{align}
with equality if and only if $U$ is the exterior domain of a round ball.
\end{thm}
\begin{rem}
The right-hand side of \eqref{eq-new} is understood as a limit when $\int_{\Sigma}H^2 da=16\pi$. In addition, the Willmore energy $\int_\Sigma H^2da$ of a closed surface $\Sigma\subset \R^3$ is known to satisfy
\begin{equation*}
\int_\Sigma H^2da\geq 16\pi,
\end{equation*}
with equality if and only if $\Sigma$ is a round sphere (see e.g. \cite{Che71,RS10,Wil68}).
\end{rem}

\begin{proof}
By \cite{HI01}, there exists a proper, Lipschitz function $\phi\geq 0$ on $\overline{U}$, called the solution to the weak inverse mean curvature flow with the initial surface $\Sigma$, satisfying the following properties:
\begin{enumerate}
  \item The function $\phi$ has value $\phi|_\Sigma=0$ and $\lim_{x\to \infty} \phi(x)=\infty$. For $t>0$, the sets $\Sigma_t=\pt\{\phi\geq t\}$ and $\Sigma_t'=\pt \{\phi >t\}$ define two increasing families of $C^{1,\alpha}$ surfaces.
  \item For $t>0$, the surfaces $\Sigma_t$ ($\Sigma_t'$, resp.) minimize (strictly minimize, resp.) area among surfaces homologous to $\Sigma_t$ in the region $\{\phi\geq t\}$. The surface $\Sigma'=\pt \{\phi>0\}$ strictly minimizes area among surfaces homologous to $\Sigma$ in $U$.
  \item For almost all $t>0$, the weak mean curvature $H$ of $\Sigma_t$ is well defined and equals $|\nabla \phi|$, which is positive for almost all $x\in \Sigma_t$.
  \item For each $t>0$, the area $|\Sigma_t|=e^t|\Sigma'|$; and $|\Sigma_t|=e^t|\Sigma|$ if $\Sigma$ is outer-minimizing (i.e., $\Sigma$ minimizes area among all surfaces homologous to $\Sigma$ in $U$).
  \item All the surfaces $\Sigma_t$ ($t>0$) remain connected. The Hawking mass
  \begin{equation}
  m_H(\Sigma_t)=\sqrt{\frac{|\Sigma_t|}{16\pi}}\left(1-\frac{1}{16\pi}\int_{\Sigma_t}H^2da_t\right)
  \end{equation}
  satisfies $\lim_{t\to 0+}m_H(\Sigma_t)\geq m_H(\Sigma')$ and its right-lower derivative satisfies (see (5.24) in \cite{HI01}; but note the misprints on some coefficients there, and the correct coefficients are as in (5.22) in \cite{HI01})
  \begin{align*}
  &\underline{D}_+m_H(\Sigma_t):=\liminf_{s\to t+}\frac{m_H(\Sigma_s)-m_H(\Sigma_t)}{s-t}\\
  &\geq \sqrt{\frac{|\Sigma_t|}{16\pi}}\frac{1}{16\pi}\left(8\pi -4\pi \chi(\Sigma_t)+\int_{\Sigma_t}(2|\nabla \log H|^2+\frac{1}{2}(\lambda_1-\lambda_2)^2)da_t\right),
  \end{align*}
  where $\lambda_1$ and $\lambda_2$ are the weak principal curvatures of $\Sigma_t$ and $\chi(\Sigma_t)$ is the Euler characteristic of $\Sigma_t$. See \cite[Section~5]{HI01} for more details.
\end{enumerate}

Now we define a modified Hawking mass
\begin{equation}
\widetilde{m}_H(\Sigma_t):=\sqrt{\frac{|\Sigma_t|}{16\pi}}m_H(\Sigma_t)=\frac{|\Sigma_t|}{16\pi}\left(1-\frac{1}{16\pi}\int_{\Sigma_t}H^2da_t\right).
\end{equation}
Then we can check $\lim_{t\to 0+}\widetilde{m}_H(\Sigma_t)\geq \widetilde{m}_H(\Sigma')$ and
\begin{align*}
&\underline{D}_+\widetilde{m}_H(\Sigma_t)
=\sqrt{\frac{|\Sigma_t|}{16\pi}}\frac{1}{2}m_H(\Sigma_t)+\sqrt{\frac{|\Sigma_t|}{16\pi}}\underline{D}_+m_H(\Sigma_t)\\
&\geq \frac{|\Sigma_t|}{(16\pi)^2}\left(\left(8\pi-\frac{1}{2}\int_{\Sigma_t}H^2da_t\right)\right.\\
&\left.\quad +8\pi -4\pi \chi(\Sigma_t)+\int_{\Sigma_t}(2|\nabla \log H|^2+\frac{1}{2}(\lambda_1-\lambda_2)^2)da_t\right)\\
&= \frac{|\Sigma_t|}{(16\pi)^2}\left(16\pi -4\pi \chi(\Sigma_t)+\int_{\Sigma_t}(2|\nabla \log H|^2-2\lambda_1 \lambda_2)da_t\right)\\
&= \frac{|\Sigma_t|}{(16\pi)^2}\left(16\pi -8\pi \chi(\Sigma_t)+\int_{\Sigma_t}2|\nabla \log H|^2da_t\right)\\
&\geq 0,
\end{align*}
where the last equality holds because of the weak Gauss--Bonnet formula (Page~403 in \cite{HI01}) and the last inequality is due to the fact the surfaces $\Sigma_t$ remain connected.
\begin{rem}
The introduction of the modified Hawking mass is inspired by the work \cite{HW15}.
\end{rem}

Next we choose the test function $f(x)=\bar{f}(\phi(x))$ for some $C^1$ function $\bar{f}:[0,\infty)\to \R$ satisfying $\bar{f}(0)=1$ and $\bar{f}(\infty)=0$ to be determined. Therefore
\begin{align*}
\mathrm{Cap}(\pt U)&\leq \int_U |\nabla f|^2dx=\int_U (\bar{f}'(\phi(x)))^2 |\nabla \phi|^2dx.
\end{align*}
Using the co-area formula, we get
\begin{align*}
\int_U (\bar{f}'(\phi(x)))^2 |\nabla \phi|^2dx&=\int_0^\infty (\bar{f}'(t))^2 \int_{\Sigma_t} |\nabla \phi|da_t dt.
\end{align*}
Note that
\begin{align*}
\int_{\Sigma_t} |\nabla \phi|da_t&=\int_{\Sigma_t}Hda_t\leq \left(\int_{\Sigma_t}H^2da_t\right)^{1/2}|\Sigma_t|^{1/2}\\
&\leq \left(16\pi -e^{-t}(16\pi-\int_{\Sigma'}H'^2 da)\right)^{1/2}e^{\frac{1}{2}t}|\Sigma'|^{1/2},
\end{align*}
where we used the H\"{o}lder inequality and the monotonicity of the modified Hawking mass. Here $H'$ denotes the mean curvature of the surface $\Sigma'$. Thus we get
\begin{align*}
\int_U |\nabla f|^2dx&\leq \int_0^\infty (\bar{f}'(t))^2 \left(16\pi e^t+(\int_{\Sigma'}H'^2 da-16\pi)\right)^{1/2} dt\cdot |\Sigma'|^{1/2}.
\end{align*}
Meanwhile, note that by the H\"{o}lder inequality, we have
\begin{align*}
1&=(\bar{f}(0))^2=(-\int_0^\infty \bar{f}'(t)dt)^2\\
&\leq \int_0^\infty (\bar{f}'(t))^2 \left(16\pi e^t+(\int_{\Sigma'}H'^2 da-16\pi)\right)^{1/2} dt\\
&\times \int_0^\infty   \left(16\pi e^t+(\int_{\Sigma'}H'^2 da-16\pi)\right)^{-1/2} dt,
\end{align*}
with equality when
\begin{equation*}
\bar{f}'(t)=c\left(16\pi e^t+(\int_{\Sigma'}H'^2 da-16\pi)\right)^{-1/2},\quad c\in \R.
\end{equation*}
Set
\begin{equation*}
s:=\frac{\int_{\Sigma'}H'^2 da}{16\pi}-1.
\end{equation*}
Note
\begin{equation*}
\int_t^\infty \left( e^{t'}+s\right)^{-1/2}dt'=\frac{2}{\sqrt{s}}\mathrm{arsinh}(\sqrt{s}e^{-t/2}).
\end{equation*}
So noticing $\bar{f}(0)=1$ and $\bar{f}(\infty)=0$ we may choose
\begin{equation*}
\bar{f}(t)=\frac{\mathrm{arsinh}(\sqrt{s}e^{-t/2})}{\mathrm{arsinh}\sqrt{s}}.
\end{equation*}
Then in this case we get
\begin{align*}
\int_U |\nabla f|^2dx&\leq  (16\pi)^{1/2}|\Sigma'|^{1/2}\frac{1}{2}\frac{\sqrt{s}}{\mathrm{arsinh}\sqrt{s}}.
\end{align*}
As a result, we have
\begin{align*}
\mathrm{Cap}(\pt U)&\leq 2 \pi^{1/2}|\Sigma'|^{1/2} \frac{\sqrt{s}}{\mathrm{arsinh}\sqrt{s}}.
\end{align*}

Next we replace $\Sigma'$ by $\Sigma$. First recall that $\Sigma'$ strictly minimizes area among all surfaces homologous to $\Sigma$. So $|\Sigma'|\leq |\Sigma|$.

Second, because $\Sigma$ is $C^2$, the surface $\Sigma'$ is $C^{1,1}$ and moreover $C^\infty$ where $\Sigma'$ does not contact $\Sigma$. Besides, the mean curvature $H'$ of $\Sigma'$ satisfies
\begin{align*}
H'=0\text{ on }\Sigma'\setminus \Sigma, \text{ and } H'=H\geq 0\text{ } a.e. \text{ on }\Sigma'\cap \Sigma.
\end{align*}
Therefore we see
\begin{equation*}
\int_{\Sigma'}H'^2 da\leq \int_\Sigma H^2 da.
\end{equation*}
In conclusion, we derive (noting that $\sqrt{s}/\mathrm{arsinh}\sqrt{s}$ is increasing in $s$)
\begin{align*}
\mathrm{Cap}(\pt U)&\leq 2(\pi |\Sigma|)^{1/2} \frac{\sqrt{\int_{\Sigma}H^2 da/(16\pi)-1}}{\mathrm{arsinh}\sqrt{\int_{\Sigma}H^2 da/(16\pi)-1}}.
\end{align*}

Finally we consider the equality case. If $\int_\Sigma H^2 da=16\pi$, then $\Sigma$ is a round sphere. Now suppose $\int_\Sigma H^2 da>16\pi$. Then checking the above proof, we see that
\begin{align*}
|\Sigma'|=|\Sigma|, \quad \int_{\Sigma'}H'^2da=\int_\Sigma H^2 da,\quad \widetilde{m}_H(\Sigma_t)=\widetilde{m}_H(\Sigma'), \; \forall t>0,
\end{align*}
and $f(x)=\bar{f}(\phi(x))$ is the electrostatic capacitary potential on $U$ with $f|_\Sigma=1$ and $f(\infty)=0$. So $\Sigma$ is outer-minimizing and the modified Hawking mass $\widetilde{m}_H(\Sigma_t)$ is equal to $\widetilde{m}_H(\Sigma)$ for all $t$. Moreover, since $f(x)$ is harmonic in $U$, any level set of $f(x)$ can not have non-empty interior, and so the surfaces $\Sigma_t$ and $\Sigma$ do not jump to $\Sigma_t'$ and $\Sigma'$ respectively, in the sense of \cite{HI01} (meaning $\Sigma_t=\Sigma_t'$ and $\Sigma=\Sigma'$). Next fix any $t>0$ and consider the exterior domain of $\Sigma_t$ in $U$. Using the fact $f(x)$ is constant on $\Sigma_t$ and $\Sigma_t$ is at least $C^1$, by the Hopf boundary lemma, we see that $\nabla f$ never vanishes on $\Sigma_t$. So $\phi=\bar{f}^{-1}\circ f$ is a smooth function on $U$ with $\nabla \phi\neq 0$. It follows that the surfaces $\{\Sigma_t\}$ evolve smoothly by the inverse mean curvature flow. Then the equality case of $\widetilde{m}_H(\Sigma_t)\geq \widetilde{m}_H(\Sigma)$ implies that $H$ is constant on $\Sigma_t$, and so $\Sigma_t$ is a round sphere. So $\Sigma$ itself is a round sphere, which contradicts $\int_\Sigma H^2 da>16\pi$. Therefore when $\int_\Sigma H^2 da>16\pi$, we have the strict inequality
\begin{align*}
\mathrm{Cap}(\pt U)&< (4\pi |\Sigma|)^{1/2} \frac{\sqrt{\int_{\Sigma}H^2 da/(16\pi)-1}}{\mathrm{arsinh}\sqrt{\int_{\Sigma}H^2 da/(16\pi)-1}}.
\end{align*}

So the proof of Theorem~\ref{thm12} is complete.
\end{proof}

\section{Proof of Theorem~\ref{thm2}}\label{sec5}

Recall the min-max variational characterization of the Steklov eigenvalues
\begin{equation*}
\xi_i=\inf_{\substack{V\subset E^1(U)\\\mathrm{dim}\,V=i}}\sup_{0\neq \varphi\in V}\frac{\int_{U}|\nabla \varphi|^2 dx}{\int_{\pt U} \varphi^2da},\quad i\geq 1.
\end{equation*}
In this section we focus on the second Steklov eigenvalue $\xi_2$.

Let $u\in E^1(U)$ be the electrostatic capacitary potential of $\pt U$. Namely, $u\in E^1(U)$ uniquely solves (Proposition~\ref{D-p})
\begin{equation*}
\begin{cases}
\Delta u=0,\text{ in }U,\\
u=1, \text{ on }\pt U.
\end{cases}
\end{equation*}
On the other hand, for any unit vector $e\in \SS^{n-1}$, there exists a unique solution, the virtual mass potential $w[e]\in \mathcal{H}(U)$ to the partial differential equation (Proposition~\ref{N-p})
\begin{equation*}
\begin{cases}
\Delta w[e]=0,\text{ in }U,\\
\dfrac{\pt w[e]}{\pt \nu}=-\dfrac{\pt \langle x,e\rangle}{\pt \nu}=-\langle \nu,e\rangle, \text{ on }\pt U.
\end{cases}
\end{equation*}
Then the function $w[e]$ and the electrostatic capacitary potential $u$ satisfy
\begin{equation}\label{eq-orthogonal}
\int_U \langle \nabla w[e],\nabla u\rangle dx=-\int_{\pt U}\frac{\pt w[e]}{\pt \nu}u\, da=\int_{\pt U}\langle \nu,e\rangle da=0.
\end{equation}
So $w[e]$ and $u$ are linearly independent, and then $\{aw[e]+bu:a,b\in \R\}$ is a $2$-dimensional linear space in $E^1(U)$.
Therefore we have
\begin{align*}
\xi_2&\leq \sup_{a^2+b^2\neq 0}\frac{\int_{U}|\nabla (aw[e]+bu)|^2 dx}{\int_{\pt U} (aw[e]+bu)^2da}\\
&= \sup_{a^2+b^2\neq 0}\frac{a^2\int_{U}|\nabla w[e]|^2 dx+b^2\int_U|\nabla u|^2dx}{a^2\int_{\pt U} w[e]^2da+b^2|\pt U|},
\end{align*}
where we used \eqref{eq-orthogonal} and the assumption in Theorem~\ref{thm2}
\begin{equation}
\int_{\pt U}w[e]uda=\int_{\pt U}w[e]da=0,\quad \forall e\in\SS^{n-1}.
\end{equation}
Choosing respectively $a=0$ and $b=0$, we conclude
\begin{align*}
\xi_2&\leq \max \big\{\frac{  \mathrm{Cap}(\pt U)}{|{\pt U}|},\frac{\int_{U}|\nabla w[e]|^2 dx}{\int_{\pt U} w[e]^2da}\big\},
\end{align*}
where we recall that $\mathrm{Cap}(\pt U)=\int_U |\nabla u|^2 dx$ is the electrostatic capacity of $\pt U$.
\begin{rem}\label{rem-Payne}
At the bottom of Page~537 in \cite{Pay56}, Payne claimed that with the proper choice of the origin the following (in our notation) holds,
\begin{align*}
\int_{\pt U}w[e]da=0,\quad \forall e\in \SS^{n-1}.
\end{align*}
We can not follow this claim, since the quantity $\int_{\pt U}w[e]da$ should be unchanged under the translation of the origin. Instead, we have to impose this assumption.
\end{rem}
Assume $\xi_2> \mathrm{Cap}(\pt U)/|{\pt U}|$. Then for any $e\in \SS^{n-1}$, we have
\begin{align*}
\xi_2&\leq \frac{(\int_{U}|\nabla w[e]|^2 dx)^2}{\int_{\pt U} w[e]^2da}\frac{1}{\int_{U}|\nabla w[e]|^2 dx}\\
&=\frac{(\int_{\pt U}\langle \nabla w[e],\nu\rangle \, w[e]da)^2}{\int_{\pt U} w[e]^2da}\frac{1}{\int_{U}|\nabla w[e]|^2 dx}\\
&\leq \frac{\int_{\pt U}\langle \nabla w[e],\nu\rangle^2 da}{\int_{U}|\nabla w[e]|^2 dx}\\
&=\frac{\int_{\pt U}\langle e,\nu\rangle^2 da}{\int_{U}|\nabla w[e]|^2 dx},
\end{align*}
which implies
\begin{equation*}
nW_{\mathrm{ave}}\, \xi_2\leq \sum_{i=1}^n \int_{\pt U}\langle \pt_i,\nu\rangle^2 da=|{\pt U}|.
\end{equation*}
Here $\pt_i$ $(i=1,2,\dots,n)$ is the standard coordinate vector and $W_{\mathrm{ave}}:=\sum_{i=1}^n \int_{U}|\nabla w[\pt_i]|^2 dx/n$ is the average virtual mass corresponding to $U$. We refer to Section~\ref{sec5.1} for its precise definition and basic properties.

By the lower bound \eqref{bound-virtual-mass1} on $W_{\mathrm{ave}}$, we conclude
\begin{equation*}
\xi_2\leq \max \big\{\frac{ \mathrm{Cap}(\pt U)}{|{\pt U}|},\frac{(n-1)|{\pt U}|}{nV}\big\}.
\end{equation*}

So we finish the proof of Theorem~\ref{thm2}.

\section{Auxiliary and related results}\label{sec6}

In this section we generalize the results in \cite{Sch57,Wal90} of the case $n=3$ to the higher dimensional case $n> 3$. Since some coefficients appearing in the generalization depends on the dimension $n$, we find it may be worth providing all the details for the convenience of readers. Our presentation mainly follows that in \cite{Wal90}. Besides, in this section the Einstein convention on the summation of indices is used unless otherwise stated, i.e., repeated indices mean the summation over them.

\subsection{Virtual mass}\label{sec5.1}

Now we consider the virtual mass of the domain $\Omega:=\R^n\setminus \overline{U}$. For any unit vector $e\in \SS^{n-1}$, there exists a unique solution, the virtual mass potential $w[e] \in E^1(U)$ to the partial differential equation (Proposition~\ref{N-p})
\begin{equation}\label{v-m-p}
\begin{cases}
\Delta w[e]=0,\text{ in }U,\\
\dfrac{\pt w[e]}{\pt \nu}=-\dfrac{\pt \langle x,e\rangle}{\pt \nu}=-\langle \nu,e\rangle, \text{ on }\pt U.
\end{cases}
\end{equation}
\begin{rem}\label{example}
It is worth noting that for any centrally symmetric exterior domain $U$, the condition in Theorem~\ref{thm2}
\begin{align}\label{eq-6.2}
\int_{\pt U}w[e]da=0,\quad \forall e\in \SS^{n-1}
\end{align}
holds. Indeed, since the definition of $w[e]$ is translation-invariant, we assume without loss of generality that $U$ is centrally symmetric with respect to the origin. Then on the one hand, by the definition and the uniqueness of the solution to \eqref{v-m-p} we see
\begin{align*}
-w[-e](x)=w[e](x),\ x\in \overline{U}.
\end{align*}
On the other hand, since $U$ is centrally symmetric, we get
\begin{align*}
w[-e](-x)=w[e](x),\ x\in \overline{U}.
\end{align*}
Combining the above two facts, we can readily derive
\begin{align*}
w[e](-x)=-w[e](x),\ x\in \overline{U},
\end{align*}
which implies \eqref{eq-6.2} immediately.
\end{rem}
\begin{defn}[\cite{SS49,AK07}]
The virtual mass matrix $W_{ij}$ ($i,j=1,2,\dots,n$) is defined as
\begin{align*}
W_{ij}:=\int_U \langle \nabla w[\pt_i],\nabla w[\pt_j]\rangle dx,
\end{align*}
where $\pt_i$ $(i=1,2,\dots,n)$ is the standard coordinate vector. The geometric invariant $\mathrm{tr}\,W/n$ is called the average virtual mass $W_{\mathrm{ave}}$.
\end{defn}

The goal of this subsection is to derive a sharp lower bound for the matrix $W$. For that purpose we need to use the gravitational potential for $\Omega=\R^n\setminus \overline{U}$. In $\R^n$ ($n\geq 3$), the gravitational potential of $\Omega$ is defined as
\begin{equation*}
\Psi(x):=-\frac{1}{(n-2)\omega_{n-1}}\int_\Omega \frac{1}{|y-x|^{n-2}}dy,\ x\in \R^n,
\end{equation*}
where $\omega_{n-1}=|\SS^{n-1}|$. By direct verification we note
\begin{align}\label{eq-harmonic}
&\Delta \Psi=0,\text{ in }U;\quad \Delta \Psi=1,\text{ in } \Omega.
\end{align}
Define
\begin{equation*}
\overline{\Psi}_{ij}:=\frac{\int_{ \Omega}\Psi_{ij}dx}{|\Omega|}.
\end{equation*}
For later use, first we have an observation.
\begin{lem}
For the matrix $\overline{\Psi}_{ij}$, we have
\begin{equation*}
0<\overline{\Psi}<I.
\end{equation*}
\end{lem}
\begin{proof}
First, using \eqref{eq-harmonic} and the divergence theorem (the integral term on the boundary of a large ball $B_R$ vanishes as $R\to \infty$ due to Proposition~\ref{prop-decay-rate}), we note
\begin{align*}
\int_{U}\Psi_{ik}\Psi_{kj}dx&=\int_{U}(\Psi_{ik}\Psi_j)_kdx\\
&=-\int_{\pt U}\Psi^{out}_{ik}\Psi_j\nu_kda\\
&=-\int_{\pt U}(\Psi^{in}_{ik}-\nu_i\nu_k)\Psi_j\nu_kda\\
&=-\int_{ \Omega}\Psi_{ik}\Psi_{kj}dx+\int_{ \Omega}\Psi_{ij} dx,
\end{align*}
where we used the relation
\begin{equation}
\Psi^{in}_{ik}(x_0)-\Psi^{out}_{ik}(x_0)=\nu_i(x_0)\nu_k(x_0), \quad x_0\in \pt U
\end{equation}
from Proposition~\ref{prop-jump} below. Here and in the sequel we use the notations
\begin{equation}
\Psi^{out}_{ij}(x_0):=\lim_{U \ni x \rightarrow x_0}\Psi_{ij}(x),\quad \Psi^{in}_{ij}(x_0):=\lim_{\Omega \ni x \rightarrow x_0}\Psi_{ij}(x).
\end{equation}
Therefore we have
\begin{align*}
\int_{\Omega}\Psi_{ij}dx=\int_{U\cup \Omega}\Psi_{ik}\Psi_{kj}dx>0.
\end{align*}

Second we note
\begin{align*}
\int_{\Omega}(\delta_{ik}-\Psi_{ik})(\delta_{kj}-\Psi_{kj})dx=\int_\Omega(\delta_{ij}-2\Psi_{ij}+\Psi_{ik}\Psi_{kj})dx.
\end{align*}
Thus we get
\begin{align*}
\int_\Omega(\delta_{ij}-\Psi_{ij})dx&=\int_{\Omega}(\delta_{ik}-\Psi_{ik})(\delta_{kj}-\Psi_{kj})dx+\int_{\Omega}\Psi_{ij}dx-\int_{\Omega}\Psi_{ik}\Psi_{kj}dx\\
&=\int_{\Omega}(\delta_{ik}-\Psi_{ik})(\delta_{kj}-\Psi_{kj})dx+\int_{U}\Psi_{ik}\Psi_{kj}dx>0.
\end{align*}
\end{proof}

Then we can prove the following sharp bound for $W$.
\begin{prop}\label{prop-virtual-mass}
For any bounded domain $\Omega\subset \R^n$ ($n\geq 3$) with smooth boundary, we have the matrix inequality
\begin{equation}\label{bound-virtual-mass}
\frac{W}{|\Omega|}\geq (I-\overline{\Psi})^{-1}-I.
\end{equation}
\end{prop}
\begin{rem}
Proposition~\ref{prop-virtual-mass} in the case $n=3$ appeared in \cite{Wal90}. We generalize here the result in \cite{Wal90} to the higher dimensional case.
\end{rem}
Taking trace on \eqref{bound-virtual-mass} and using the Cauchy--Schwarz inequality, we obtain the following.
\begin{corr}
For any bounded domain $\Omega\subset \R^n$ ($n\geq 3$) with smooth boundary, we have
\begin{equation}\label{bound-virtual-mass1}
W_{\mathrm{ave}}\geq \frac{1}{n-1}|\Omega|.
\end{equation}
\end{corr}
\begin{rem}
When $n=3$, the bound \eqref{bound-virtual-mass1} is due to \cite{Sch57}.
\end{rem}
\begin{rem}
In the literature, there are extensive works concerning the generalized polarization which includes the virtual mass (and the polarization in Section~\ref{sec5.2}) as a limit case. The generalized polarization (see e.g. \cite{AK07}) defined for a bounded set $\Omega\subset \R^n$ involves a parameter $0<k\neq 1<\infty$, the conductivity of $\Omega$. When $k=0$ ($\Omega$ is insulated), the generalized polarization reduces to the virtual mass; while when $k\rightarrow \infty$ ($\Omega$ is perfectly conducting), it reduces to the polarization in Section~\ref{sec5.2}. See \cite[Page~89]{AK07}. The inequality for $0<k\neq 1<\infty$ corresponding to \eqref{bound-virtual-mass1} was obtained in \cite{BK93,Lip93,CV03}; see \cite{CV04} for a survey. In \cite{BK93,Lip93,CV03}, different methods from the one in the proof of Proposition~\ref{prop-virtual-mass} were used. Moreover, the equality (rigidity) case for $0<k\neq 1<\infty$ and $n=3$ was handled in \cite{KM08}.
\end{rem}
\begin{proof}[Proof of Proposition~\ref{prop-virtual-mass}]

%
%

For a constant symmetric matrix $A=[a_{ij}]$ to be determined, define a family of functions $\widetilde{w}_i$ by
\begin{equation*}
\widetilde{w}_i(x)=a_{ij}\Psi_j(x).
\end{equation*}
We have the matrix inequality
\begin{equation*}
\int_U\langle \nabla (w[\pt_i]-\widetilde{w}_i), \nabla (w[\pt_j]-\widetilde{w}_j)\rangle dx\geq 0,
\end{equation*}
which leads to
\begin{align*}
W_{ij}&\geq \int_U (\langle \nabla w[\pt_i],\nabla \widetilde{w}_j\rangle+\langle \nabla w[\pt_j],\nabla \widetilde{w}_i\rangle-\langle \nabla \widetilde{w}_i,\nabla \widetilde{w}_j\rangle)dx\\
&=\int_{\pt \Omega}\left(\frac{\pt \widetilde{w}_i}{\pt \nu}\widetilde{w}_j-\frac{\pt w[\pt_i]}{\pt \nu}\widetilde{w}_j-\frac{\pt w[\pt_j]}{\pt \nu}\widetilde{w}_i\right)da\\
&=\int_{\pt \Omega}\left(a_{ip}a_{jq}\frac{\pt \Psi^{out}_{p}}{\pt \nu}\Psi_q+\langle \nu,\pt_i\rangle a_{jq}\Psi_q+\langle \nu,\pt_j\rangle a_{ip}\Psi_p\right)da\\
&=a_{ip}a_{jq}\int_{\pt \Omega}\frac{\pt \Psi^{out}_{p}}{\pt \nu}\Psi_q da+a_{jq}\int_\Omega \Psi_{iq}dx+a_{ip}\int_\Omega \Psi_{jp}dx.
\end{align*}
For the first term, we use Proposition~\ref{prop-jump} below to get
\begin{align*}
\int_{\pt \Omega}\frac{\pt \Psi^{out}_{p}}{\pt \nu}\Psi_q da&=\int_{\pt \Omega}\Psi^{out}_{pk}\nu_k\,\Psi_q da=\int_{\pt \Omega}(\Psi^{in}_{pk}-\nu_p\nu_k)\nu_k \,\Psi_q da\\
&=\int_{\pt \Omega}\frac{\pt \Psi^{in}_{p}}{\pt \nu}\Psi_q da-\int_{\pt \Omega}\langle  \Psi_q\,\pt_p,\nu\rangle da\\
&=\int_{ \Omega}\Psi_{pk}\Psi_{qk}dx-\int_{ \Omega}\Psi_{pq}dx.
\end{align*}
Recall
\begin{equation*}
\overline{\Psi}_{ij}:=\frac{\int_{ \Omega}\Psi_{ij}dx}{|\Omega|}.
\end{equation*}
Note that we have the matrix inequality
\begin{equation*}
\int_\Omega (\Psi_{ik}-\overline{\Psi}_{ik})(\Psi_{jk}-\overline{\Psi}_{jk})dx\geq 0,
\end{equation*}
which is the same as
\begin{equation*}
\int_{ \Omega}\Psi_{pk}\Psi_{qk}dx\geq \overline{\Psi}_{pk}\overline{\Psi}_{qk}|\Omega|.
\end{equation*}
Therefore we conclude
\begin{align*}
\frac{W_{ij}}{|\Omega|}&\geq a_{ip}a_{jq}(\overline{\Psi}_{pk}\overline{\Psi}_{qk}-\overline{\Psi}_{pq})+a_{jq}\overline{\Psi}_{iq}+a_{ip}\overline{\Psi}_{jp}.
\end{align*}
Now we choose
\begin{equation*}
A=(I-\overline{\Psi})^{-1}.
\end{equation*}
So we get the matrix inequality
\begin{equation*}
\frac{W}{|\Omega|}\geq (I-\overline{\Psi})^{-1}-I.
\end{equation*}
\end{proof}

\subsection{Polarization}\label{sec5.2}

As a byproduct, let us consider the parallel results for the polarization of those in Section~\ref{sec5.1}. For any unit vector $e\in \SS^{n-1}$, there exists a unique solution, the polarization potential $v[e]\in E^1(U)$ (together with a unique constant $c[e]$) to the partial differential equation (Proposition~\ref{D-p})
\begin{equation*}
\begin{cases}
\Delta v[e]=0,\text{ in }U,\\
v[e]=\langle x,e\rangle +c[e], \text{ on }\pt U, \\
v[e]=O(|x|^{1-n}),\text{ as }x\rightarrow \infty.
\end{cases}
\end{equation*}
\begin{defn}[\cite{SS49,AK07}]
The polarization matrix $P_{ij}$ ($i,j=1,2,\dots,n$) is defined as
\begin{align*}
P_{ij}:=\int_U \langle \nabla v[\pt_i],\nabla v[\pt_j]\rangle dx.
\end{align*}
The geometric invariant $\mathrm{tr}\,P/n$ is called the average polarization $P_{\mathrm{ave}}$.
\end{defn}

Now we prove the following sharp bound for $P$.
\begin{prop}\label{prop-polarization}
For any bounded domain $\Omega\subset \R^n$ ($n\geq 3$) with smooth boundary, we have
\begin{equation}
\frac{P}{|\Omega|}\geq \overline{\Psi}^{-1}-I.
\end{equation}
\end{prop}
\begin{rem}
Proposition~\ref{prop-polarization} in the case $n=3$ appeared in \cite{Wal90}. Again we generalize here the result in \cite{Wal90} to the higher dimensional case.
\end{rem}
Taking trace and using the Cauchy--Schwarz inequality, we get the following.
\begin{corr}
For any bounded domain $\Omega\subset \R^n$ ($n\geq 3$) with smooth boundary, we have
\begin{equation}
P_{\mathrm{ave}}\geq (n-1)|\Omega|.
\end{equation}
\end{corr}

\begin{proof}[Proof of Proposition~\ref{prop-polarization}]

For a constant symmetric matrix $B=[b_{ij}]$ to be determined, define a family of functions $\widetilde{v}_i$ by
\begin{equation*}
\widetilde{v}_i(x)=b_{ij}\Psi_j(x).
\end{equation*}
We have the matrix inequality
\begin{equation*}
\int_U\langle \nabla (v[\pt_i]-\widetilde{v}_i), \nabla (v[\pt_j]-\widetilde{v}_j)\rangle dx\geq 0,
\end{equation*}
which leads to
\begin{align*}
P_{ij}&\geq \int_U (\langle \nabla v[\pt_i],\nabla \widetilde{v}_j\rangle+\langle \nabla v[\pt_j],\nabla \widetilde{v}_i\rangle-\langle \nabla \widetilde{v}_i,\nabla \widetilde{v}_j\rangle)dx\\
&=\int_{\pt \Omega}\left(\frac{\pt \widetilde{v}_i}{\pt \nu}\widetilde{v}_j-\frac{\pt \widetilde{v}_j}{\pt \nu}v[\pt_i]-\frac{\pt \widetilde{v}_i}{\pt \nu}v[\pt_j]\right)da\\
&=\int_{\pt \Omega}\left(b_{ip}b_{jq}\frac{\pt \Psi^{out}_{p}}{\pt \nu}\Psi_q- b_{jq}\frac{\pt \Psi^{out}_q}{\pt \nu}(\langle x,\pt_i\rangle+c[\pt_i])\right.\\
&\quad \left.- b_{ip}\frac{\pt \Psi^{out}_p}{\pt \nu}(\langle x,\pt_j\rangle+c[\pt_j])\right)da\\
&=b_{ip}b_{jq}\int_{\pt \Omega}\frac{\pt \Psi^{out}_{p}}{\pt \nu}\Psi_q da- b_{jq}\int_{\pt \Omega}\frac{\pt \Psi^{out}_q}{\pt \nu}\langle x,\pt_i\rangle da\\
&\quad - b_{ip}\int_{\pt \Omega}\frac{\pt \Psi^{out}_p}{\pt \nu}\langle x,\pt_j\rangle da.
\end{align*}
For the first term, we use Proposition~\ref{prop-jump} below to get
\begin{align*}
\int_{\pt \Omega}\frac{\pt \Psi^{out}_{p}}{\pt \nu}\Psi_q da&=\int_{\pt \Omega}\Psi^{out}_{pk}\nu_k\,\Psi_q da=\int_{\pt \Omega}(\Psi^{in}_{pk}-\nu_p\nu_k)\nu_k \,\Psi_q da\\
&=\int_{\pt \Omega}\frac{\pt \Psi^{in}_{p}}{\pt \nu}\Psi_q da-\int_{\pt \Omega}\langle  \Psi_q\,\pt_p,\nu\rangle da\\
&=\int_{ \Omega}\Psi_{pk}\Psi_{qk}dx-\int_{ \Omega}\Psi_{pq}dx\\
&\geq \overline{\Psi}_{pk}\overline{\Psi}_{qk}|\Omega|-\int_{ \Omega}\Psi_{pq}dx,
\end{align*}
where the inequality came from the proof of Proposition~\ref{prop-virtual-mass}. For the second term, we get
\begin{align*}
\int_{\pt \Omega}\frac{\pt \Psi^{out}_q}{\pt \nu}\langle x,\pt_i\rangle da&=\int_{\pt \Omega}(\Psi^{in}_{qk}-\nu_q\nu_k)\nu_k\langle x,\pt_i\rangle da\\
&=\int_{\pt \Omega}\frac{\pt \Psi^{in}_q}{\pt \nu}\langle x,\pt_i\rangle da-\int_{\pt \Omega}\langle \nu,\pt_q\rangle \langle x,\pt_i\rangle da\\
&=\int_\Omega \Psi_{iq}dx-\delta_{iq}|\Omega|.
\end{align*}
And we handle the third term similarly.

Therefore we conclude
\begin{align*}
\frac{P_{ij}}{|\Omega|}&\geq b_{ip}b_{jq}(\overline{\Psi}_{pk}\overline{\Psi}_{qk}-\overline{\Psi}_{pq})-b_{jq}(\overline{\Psi}_{iq}-\delta_{iq})-b_{ip}(\overline{\Psi}_{jp}-\delta_{jp}).
\end{align*}
Now we choose
\begin{equation*}
B=\overline{\Psi}^{-1}.
\end{equation*}
So we get the matrix inequality
\begin{equation*}
\frac{P}{|\Omega|}\geq \overline{\Psi}^{-1}-I.
\end{equation*}
\end{proof}

\subsection{Gravitational potential}\label{sec5.3}

In $\R^n$ ($n\geq 3$), define the gravitational potential as
\begin{equation*}
\Psi(x):=-\frac{1}{(n-2)\omega_{n-1}}\int_\Omega \frac{1}{|y-x|^{n-2}}dy,
\end{equation*}
where $\omega_{n-1}=|\SS^{n-1}|$.

This subsection is devoted to proving the following.
\begin{prop}\label{prop-jump}
Let $\Omega\subset \R^n$ $(n\geq 3)$ be a bounded domain in $\R^n$ with smooth boundary. Fix $x_0\in \pt \Omega$. For the second derivatives of $\Psi(x)$, we have
\begin{equation}\label{eq-limit}
\lim_{\Omega \ni x \rightarrow x_0}\Psi_{ij}(x)-\lim_{U \ni x \rightarrow x_0}\Psi_{ij}(x)=\nu_i(x_0)\nu_j(x_0).
\end{equation}
\end{prop}
\begin{proof}
Note that for $y\neq x$, we have
\begin{equation*}
\mathrm{div}_y\big(\frac{y-x}{|y-x|^{n-2}}\big)=\frac{2}{|y-x|^{n-2}}.
\end{equation*}
So for $x\in U$, we get
\begin{align*}
\Psi(x)&=-\frac{1}{2(n-2)\omega_{n-1}}\int_\Omega \mathrm{div}_y\big(\frac{y-x}{|y-x|^{n-2}}\big)dy\\
&=-\frac{1}{2(n-2)\omega_{n-1}}\int_{\pt \Omega} \frac{\langle y-x,\nu\rangle }{|y-x|^{n-2}}da(y);
\end{align*}
while for $x\in \Omega$, by choosing a small ball $B_r(x)\subset \Omega$, we get
\begin{align*}
\Psi(x)&=-\frac{1}{(n-2)\omega_{n-1}}\lim_{r\rightarrow 0+}\int_{\Omega\setminus B_r(x)} \frac{1}{|y-x|^{n-2}}dy\\
&=-\frac{1}{2(n-2)\omega_{n-1}}\lim_{r\rightarrow 0+}\int_{\Omega\setminus B_r(x)} \mathrm{div}_y\big(\frac{y-x}{|y-x|^{n-2}}\big)dy\\
&=-\frac{1}{2(n-2)\omega_{n-1}}\lim_{r\rightarrow 0+}\left(\int_{\pt \Omega} \frac{\langle y-x,\nu\rangle }{|y-x|^{n-2}}da(y)-\int_{\pt B_r(x)} \frac{\langle y-x,\nu\rangle }{|y-x|^{n-2}}da(y)\right)\\
&=-\frac{1}{2(n-2)\omega_{n-1}}\int_{\pt \Omega} \frac{\langle y-x,\nu\rangle }{|y-x|^{n-2}}da(y).
\end{align*}
In either case, we have
\begin{align*}
\Psi(x)&=-\frac{1}{2(n-2)\omega_{n-1}}\int_{\pt \Omega} \frac{\langle y-x,\nu\rangle }{|y-x|^{n-2}}da(y).
\end{align*}
Then we obtain its first derivatives
\begin{align*}
\Psi_i(x)&=-\frac{1}{2(n-2)\omega_{n-1}}\int_{\pt \Omega} \left(\frac{-\nu_i}{|y-x|^{n-2}}+(n-2)\frac{\langle y-x,\nu\rangle (y_i-x_i)}{|y-x|^{n}}\right)da(y),
\end{align*}
and its second derivatives
\begin{align*}
\Psi_{ij}(x)&=-\frac{1}{2\omega_{n-1}}\int_{\pt \Omega} \bigg(\frac{\nu_i(x_j-y_j)+\nu_j(x_i-y_i)}{|y-x|^n}-\delta_{ij}\frac{\langle y-x,\nu\rangle }{|y-x|^n}\\
&+n\langle y-x,\nu\rangle \frac{(x_i-y_i)(x_j-y_j)}{|y-x|^{n+2}}\bigg)da(y)\\
&=:-\frac{1}{2\omega_{n-1}}\int_{\pt \Omega}\langle \nu,X_{ij}(y)\rangle da(y),
\end{align*}
where we defined a vector field $X_{ij}(y)$ for $y\neq x$ as
\begin{align*}
X_{ij}(y)&=\frac{(x_j-y_j)\pt_i+(x_i-y_i)\pt_j}{|y-x|^n}-\delta_{ij}\frac{y-x}{|y-x|^n}+n (y-x) \frac{(x_i-y_i)(x_j-y_j)}{|y-x|^{n+2}}.
\end{align*}

Note that for $y\neq x$, by direct computation we have
\begin{align*}
&\mathrm{div}_y(X_{ij}(y))=2\left(n\frac{(y_i-x_i)(y_j-x_j)}{|y-x|^{n+2}}- \frac{\delta_{ij}}{|y-x|^n}\right).
\end{align*}
So for $x\in U$, we can conclude
\begin{align*}
\Psi_{ij}(x)&=-\frac{1}{\omega_{n-1}}\int_{ \Omega} \left(n\frac{(y_i-x_i)(y_j-x_j)}{|y-x|^{n+2}}- \frac{\delta_{ij}}{|y-x|^n}\right)dy;
\end{align*}
while for $x\in \Omega$, by choosing again a small ball $B_r(x)\subset \Omega$, we get
\begin{align*}
\Psi_{ij}(x)&=-\frac{1}{2\omega_{n-1}}\int_{\pt \Omega}\langle \nu,X_{ij}(y)\rangle da(y)\\
&=-\frac{1}{2\omega_{n-1}}\left(\int_{ \Omega\setminus B_r(x)}\mathrm{div}_y(X_{ij}(y))dy+\int_{\pt B_r(x)}\langle \nu,X_{ij}(y)\rangle da(y)\right).
\end{align*}
Direct computation yields
\begin{align*}
&\int_{\pt B_r(x)}\langle \nu,X_{ij}(y)\rangle da(y)=-\frac{2}{n}\omega_{n-1}\delta_{ij}.
\end{align*}
So for $x\in \Omega$ we get
\begin{align*}
\Psi_{ij}(x)&=-\frac{1}{\omega_{n-1}}\int_{ \Omega\setminus B_r(x)}\left(n\frac{(y_i-x_i)(y_j-x_j)}{|y-x|^{n+2}}- \frac{\delta_{ij}}{|y-x|^n}\right)dy+\frac{1}{n}\delta_{ij}\\
&=-\frac{1}{\omega_{n-1}}\int_{ \Omega}\left(n\frac{(y_i-x_i)(y_j-x_j)}{|y-x|^{n+2}}- \frac{\delta_{ij}}{|y-x|^n}\right)dy+\frac{1}{n}\delta_{ij},
\end{align*}
where in the last step we took $r\rightarrow 0+$.

Without loss of generality, we choose $x_0$ as the origin and $\nu$ as the positive $x_n$-direction.

Note first that for $r>0$ sufficiently small, the intersection set $\Omega \cap B_r(x_0)$ is approximately a half $n$-ball $B^-_r(x_0)$. Fix a small $r>0$. Then we intend to compare the integrals in the expression of $\Psi_{ij}(x)$ over $\Omega \cap B_r(x)$ and over $\Omega \setminus B_r(x)$ for $x\in U$ or $x\in \Omega$.

For the integral over $\Omega \setminus B_r(x)$, it is continuous with respect to $x$. It remains to consider the integral over $\Omega \cap B_r(x)$, namely,
\begin{equation*}
I_{ij}(x):=\int_{ \Omega\cap B_r(x)}f_{ij}(x,y)dy,
\end{equation*}
where
\begin{equation*}
f_{ij}(x,y):=n\frac{(y_i-x_i)(y_j-x_j)}{|y-x|^{n+2}}- \frac{\delta_{ij}}{|y-x|^n}.
\end{equation*}

When $i\neq j$, since
\begin{equation*}
\int_{B^-_r(x_0)}\left(n\frac{(y_i-(x_0)_i)(y_j-(x_0)_j)}{|y-x_0|^{n+2}}\right)dy=0,
\end{equation*}
both $\lim_{\Omega \ni x \rightarrow x_0}I_{ij}(x)$ and $\lim_{U \ni x \rightarrow x_0}I_{ij}(x)$ can be as close to zero as possible, as long as $r$ is chosen at first very small. So we finish the proof of \eqref{eq-limit} for $i\neq j$.

Next we consider $i=j$. It suffices to consider the two cases $i=j=1$ and $i=j=n$. Define the approximating quantities
\begin{equation*}
\widetilde{I}_{ij}(x):=\int_{ \{y\in \R^n|y_n<0\}\cap B_r(x)}\left(n\frac{(y_i-x_i)(y_j-x_j)}{|y-x|^{n+2}}- \frac{\delta_{ij}}{|y-x|^n}\right)dy.
\end{equation*}
We only need to compute
\begin{align*}
&A_-:=\lim_{\R^n_-\ni x\rightarrow 0}\widetilde{I}_{11}(x),\quad
A_+:=\lim_{\R^n_+\ni x\rightarrow 0}\widetilde{I}_{11}(x),\\
&B_-:=\lim_{\R^n_-\ni x\rightarrow 0}\widetilde{I}_{nn}(x),\quad
B_+:=\lim_{\R^n_+\ni x\rightarrow 0}\widetilde{I}_{nn}(x),
\end{align*}
where $\R^n_-=\R^n\cap \{x_n<0\}$ and $\R^n_+=\R^n\cap \{x_n>0\}$.

Note that we have the following identities,
\begin{align*}
A_-+A_+&=\int_{B_r(0)} \left(n\frac{y_1^2}{|y|^{n+2}}- \frac{1}{|y|^n}\right)dy\\
&=\int_{B_r(0)} \left(\frac{\sum_{k=1}^ny_k^2}{|y|^{n+2}}- \frac{1}{|y|^n}\right)dy=0,\\
B_-+B_+&=\int_{B_r(0)} \left(n\frac{y_n^2}{|y|^{n+2}}- \frac{1}{|y|^n}\right)dy=0,\\
(n-1)A_-+B_-&=\int_{\R^n_-\cap B_r(0)} \left(n\frac{|y|^2}{|y|^{n+2}}- \frac{n}{|y|^n}\right)dy=0.
\end{align*}
So we only need to compute one of them. Let us compute $B_+$. Equivalently, we derive
\begin{align*}
B_+&=\lim_{\varepsilon \rightarrow 0+}\int_{B_r(0)\cap \{y_n>\varepsilon\}}\left(n\frac{y_n^2}{|y|^{n+2}}- \frac{1}{|y|^n}\right)dy\\
&=\lim_{\varepsilon \rightarrow 0+}\int_\varepsilon^rdy_n \int_{|x'|<\sqrt{r^2-y_n^2}}\left(n\frac{y_n^2}{(|x'|^2+y_n^2)^{(n+2)/2}}-\frac{1}{(|x'|^2+y_n^2)^{n/2}}\right)dx'\\
&=\lim_{\varepsilon \rightarrow 0+}\int_\varepsilon^rdy_n \int_0^{\sqrt{r^2-y_n^2}}\left(n\frac{y_n^2}{(s^2+y_n^2)^{(n+2)/2}}-\frac{1}{(s^2+y_n^2)^{n/2}}\right)\omega_{n-2}s^{n-2}ds.
\end{align*}
Let $s=y_n\sinh t$. Then we have
\begin{align*}
B_+&=\lim_{\varepsilon \rightarrow 0+}\int_\varepsilon^r\frac{\omega_{n-2}}{y_n}dy_n \int_0^{\mathrm{arsinh}(\sqrt{r^2-y_n^2}/y_n)}\frac{n-\cosh^2 t}{\cosh^{n+1}t}\sinh^{n-2}t\,dt\\
&=\lim_{\varepsilon \rightarrow 0+}\int_\varepsilon^r\frac{\omega_{n-2}}{y_n}dy_n \,\frac{\sinh^{n-1}t}{\cosh^n t}\bigg|_0^{\mathrm{arsinh}(\sqrt{r^2-y_n^2}/y_n)}\\
&=\lim_{\varepsilon \rightarrow 0+}\int_\varepsilon^r\frac{\omega_{n-2}}{y_n}\frac{(\sqrt{r^2-y_n^2}/y_n)^{n-1}}{(1+(\sqrt{r^2-y_n^2}/y_n)^2)^{n/2}}dy_n\\
&=\omega_{n-2}\lim_{\varepsilon \rightarrow 0+}\int_\varepsilon^rr^{-n}(r^2-y_n^2)^{(n-1)/2}dy_n\\
&=\omega_{n-2}\int_0^{\pi/2}\sin^n t\, dt=\omega_{n-2}\frac{n-1}{n}\int_0^{\pi/2}\sin^{n-2} t\, dt\\
&=\frac{n-1}{2n}\omega_{n-1}.
\end{align*}
Then we can deduce
\begin{align*}
B_-=-\frac{n-1}{2n}\omega_{n-1},\quad A_-=\frac{1}{2n}\omega_{n-1},\quad A_+=-\frac{1}{2n}\omega_{n-1}.
\end{align*}
Now we are ready to prove the remaining conclusion. Let $\eta>0$ be arbitrary. Then we can choose $r$ small enough such that
\begin{align*}
&|\lim_{U\ni x\rightarrow x_0}\int_{B_r(x)\cap \Omega}f_{nn}(x,y)dy-B_+|<\omega_{n-1}\frac{\eta }{2},\\
&|\lim_{\Omega \ni x\rightarrow x_0}\int_{B_r(x)\cap \Omega}f_{nn}(x,y)dy-B_-|<\omega_{n-1}\frac{\eta }{2}.
\end{align*}
Note that
\begin{align*}
\lim_{U\ni x\rightarrow x_0}\Psi_{nn}(x)&=-\frac{1}{\omega_{n-1}}\int_{\Omega\setminus B_r(x_0)}f_{nn}(x_0,y)dy\\
&-\frac{1}{\omega_{n-1}}\lim_{U\ni x\rightarrow x_0}\int_{B_r(x)\cap \Omega}f_{nn}(x,y)dy,\\
\lim_{\Omega\ni x\rightarrow x_0}\Psi_{nn}(x)&=-\frac{1}{\omega_{n-1}}\int_{\Omega\setminus B_r(x_0)}f_{nn}(x_0,y)dy\\
&-\frac{1}{\omega_{n-1}}\lim_{\Omega\ni x\rightarrow x_0}\int_{B_r(x)\cap \Omega}f_{nn}(x,y)dy+\frac{1}{n}.
\end{align*}
So we obtain
\begin{align*}
&|\lim_{\Omega\ni x\rightarrow x_0}\Psi_{nn}(x)-\lim_{U\ni x\rightarrow x_0}\Psi_{nn}(x)-1|\\
&< |-\frac{1}{\omega_{n-1}}(B_--B_+)+\frac{1}{n}-1|+\eta=\eta,
\end{align*}
which implies
\begin{align*}
\lim_{\Omega\ni x\rightarrow x_0}\Psi_{nn}(x)=\lim_{U\ni x\rightarrow x_0}\Psi_{nn}(x)+1.
\end{align*}
Similar arguments show that
\begin{align*}
\lim_{\Omega\ni x\rightarrow x_0}\Psi_{ii}(x)=\lim_{U\ni x\rightarrow x_0}\Psi_{ii}(x),\quad i=1,2,\dots,n-1.
\end{align*}
So we finish the proof of Proposition~\ref{prop-jump}.

\end{proof}

\bibliographystyle{Plain}

\end{document}